\documentclass[12pt]{amsart}
\newcounter{dummy}
\usepackage{amsmath,amsthm,amscd,amsfonts,wasysym,amssymb,epic,eepic,bbm,tikz-cd,mathtools,multicol,enumitem,mathrsfs, enumitem, comment}
\usepackage[pagebackref,colorlinks=true,linkcolor=blue,citecolor=blue]{hyperref}
\usepackage[noabbrev]{cleveref}
\usepackage{MnSymbol}
\makeatletter
\newcommand\myitem[1][]{\item[#1]\refstepcounter{dummy}\def\@currentlabel{#1}}
\makeatother



\allowdisplaybreaks
\newtheorem{thm}{Theorem}[section]

\newtheorem{lem}[thm]{Lemma}

\theoremstyle{remark}
\newtheorem*{rem}{Remark}

\newcounter{remarkscounter}

\numberwithin{equation}{section}
\newcommand{\A}{\mathbb{A}}
\newcommand{\GL}{\mathrm{GL}}

\newcommand{\ZZ}{\mathbb{Z}}

\newcommand{\QQ}{\mathbb{Q}}

\newcommand{\lto}{\longrightarrow}

\newcommand{\OO}{\mathcal{O}}
\newcommand{\CC}{\mathbb{C}}
\newcommand{\RR}{\mathbb{R}}
\newcommand{\GG}{\mathbb{G}}

\newcommand{\quash}[1]{}

\theoremstyle{definition}
\newtheorem{defn}[thm]{Definition}

\numberwithin{equation}{subsection}

\theoremstyle{definition}

\theoremstyle{remark}

\allowdisplaybreaks

\makeindex

\begin{document}

\title{Triple product $L$-functions and the Ramanujan conjecture}

\author{Jayce R. Getz}
\address{Department of Mathematics\\
Duke University\\
Durham, NC 27708}
\email{jayce.getz@duke.edu}

\author{Heekyoung Hahn}
\address{Department of Mathematics\\
Duke University\\
Durham, NC 27708}
\email{heekyoung.hahn@duke.edu}

\author{HaoYun Yao}
\address{Department of Mathematics\\
Duke University\\
Durham, NC 27708}
\email{haoyun.yao@duke.edu}


\thanks{The first author is thankful for partial support provided by NSF grant DMS-2400550 and NSF RTG grant DMS-2231514. The second author appreciates the partial support by NSF RTG grant DMS-2231514. Any opinions, findings, and conclusions or recommendations expressed in this material are those of the authors and do not necessarily reflect the views of the National Science Foundation. 
}

\begin{abstract}
We prove that the Ramanujan conjecture is true under the assumption that the expected analytic properties of triple product $L$-functions hold.  Further, we explain how these analytic properties imply certain reduction steps in the construction of functorial transfers in the sense of Langlands.  Roughly, at the level of stably automorphic representations, they allow one to reduce any functorial transfer from a given reductive group $G$ to a general linear group to a finite family of transfers depending on $G.$
\end{abstract}

\maketitle

\section{Introduction}

Let $G$ be a reductive group over a number field $F$ and let $\A_F$ be the ring of  adeles of $F.$ For a given representation
\begin{align}\label{Lmap}
r: {}^LG\lto \GL_n(\CC),
\end{align} 
the Langlands functoriality conjecture predicts there should be a corresponding transfer of automorphic representations of $G(\A_F)$ to automorphic representations of $\GL_n(\A_F)$.

In \cite{BK-lifting}, Braverman and Kahzdan suggested a method for establishing this transfer (see also \cite{LafforgueJJM, NgoSums,Ngo:Hankel}).  First they propose  certain conjectural Poisson summation formulae for reductive monoids attached to $r.$  By a generalization of the method of Godement and Jacquet \cite{GodementJacquetBook}, these Poisson summation formulae can be used to prove the analytic properties of the Langlands $L$-functions $L(s,\pi,r).$  They then suggest that the converse theorem can be used to derive the desired transfer.  

For any $n_1,n_2,n_3 \in \ZZ_{\geq 1},$ let 
\begin{align}\label{tensor3}
\otimes^3 : {}^L(\GL_{n_1}\times \GL_{n_2} \times \GL_{n_3}) \longrightarrow \GL_{n_1n_2n_3}(\CC)
\end{align}
be the triple tensor product representation.  Moreover let $\pi_i$ be a cuspidal automorphic representation of $\GL_{n_i}(\A_F)$ for $i\in \{1, 2, 3\}.$  
In the recent joint work \cite{Getz:Gu:Hsu:Leslie}, the first author constructed an integral representation related to the triple product $L$-functions  
\begin{equation}\label{defn:triple-Lfn}
L(s,\pi_1\times\pi_2\times\pi_3):=L(s,\pi_1\otimes\pi_2\otimes\pi_3, \otimes^3).
\end{equation}
These are the Langlands $L$-functions associated to the triple tensor product.
We hasten to point out that the analytic properties of these $L$-functions are still not known. The paper \cite{Getz:Gu:Hsu:Leslie} gives a proposal for reducing these analytic properties to local questions.  

Motivated by this work, in the current paper we investigate what can be derived about Langlands functoriality from the knowledge of the analytic properties of triple product $L$-functions.  Our primary tool is a converse theorem of Cogdell and Piatetski-Shapiro (see Theorem \ref{Thm:3CPS} below).  

Given irreducible admissible representations $\pi_1$ of $\GL_{n_1}(\A_F)$ and $\pi_2$ of $\GL_{n_2}(\A_F)$ one can define the Rankin-Selberg transfer
\begin{align}
    \pi_1 \boxtimes \pi_2
\end{align}
of $\pi_1 \otimes \pi_2.$ It is an admissible representation of $\GL_{n_1n_2}(\A_F).$  We give more details in \S \ref{sec:conv:app} below.  We emphasize that $\pi_1 \boxtimes \pi_2$ and $\pi_1 \otimes \pi_2$ are very different objects. The former is a representation of $\GL_{n_1n_2}(\A_F)$ and the latter is a representation of $\GL_{n_1}(\A_F) \times \GL_{n_2}(\A_F).$ 

\begin{defn}\label{IU}
An automorphic representation of $\GL_n(\A_F)$ is of \textbf{IU-type} if it is an isobaric sum of unitary cuspidal automorphic representations.
\end{defn}

Let $S$ be a finite set of places of $F$ including the infinite places. 

\begin{thm}\label{thm:RS:intro}
Let $\pi_i, i\in \{1,2\},$ be unitary cuspidal automorphic representations of $\GL_{n_i}(\A_F)$ unramified outside $S.$  Assume that
\begin{enumerate}[label=(\alph*)]
\item{For all $1 \leq n_3 < n_1n_2$ and all unitary cuspidal automorphic representations $\pi_3$ of $\GL_{n_3}(\A_F)$ that are ramified at some place not in $S,$ the triple product $L$-function
$L(s,\pi_1\times\pi_2\times\pi_3)$ is nice.}\label{assum:nice:intro}
\item{For all $1 \leq n_3 <n_1n_2$ and all unitary cuspidal automorphic representations $\pi_3$ of $\GL_{n_3}(\A_F),$ the triple product $L$-functions
$L(s,\pi_1\times\pi_2\times\pi_3)$ and $L(s,\pi_1^\vee \times \pi_2^\vee \times \pi_3)$ admit analytic continuations to $\mathrm{Re}(s)>1.$}\label{assum:pole:intro} 
\end{enumerate}
Then there exists a weak Rankin-Selberg transfer of  $\pi_1 \otimes \pi_2.$  It is of IU-type and is compatible at all finite places where both $\pi_1$ and $\pi_2$ are unramified. 
\end{thm}

\noindent The notion of ``nice" is the same as that of Cogdell and Piatetski-Shapiro (see Definition \ref{defn:nice}). Briefly, nice $L$-functions are those that are bounded in vertical strips, admit the usual functional equation, and are holomorphic.  

The assumptions in Theorem \ref{thm:RS:intro} may seem strange if one has not thought about general Langlands transfers before.  The root of this is that the converse theorems available in the literature are adapted to study functorial transfers that send cuspidal automorphic representations to cuspidal automorphic representations.  In contrast, Rankin-Selberg transfers do not always send cuspidal automorphic representations to cuspidal automorphic representations.  In particular 
$L(s,\pi_1\times\pi_2\times\pi_3)$  will not be nice in general because it need not be holomorphic.  Thus we cannot apply the converse theorem of Cogdell and Piatetski-Shapiro, Theorem \ref{Thm:3CPS}, without additional argument.  
As explained in the remark at the end of \S \ref{sec:Ramanujan:app}, a sufficiently robust formulation of Langlands functoriality implies that the assumptions \ref{assum:nice:intro} and \ref{assum:pole:intro} in Theorem \ref{thm:RS:intro} are valid.  Moreover, they are the sort of information one expects will be accessible from an integral representation of $L(s,\pi_1 \times \pi_2 \times \pi_3).$

As one indication of the power of this particular case of Langlands functoriality, we modify a classic argument of Langlands \cite{Langlands:Problems} to show that it implies the Ramanujan conjecture:

\begin{thm}\label{thm:Ramanujan:intro}
Suppose that for all $n_1, n_2 \in \ZZ_{\geq 1}$ and all unitary cuspidal automorphic representations $\pi_i$ of $\GL_{n_i}(\A_F), i=1,2,$ a weak Rankin-Selberg transfer of $\pi_1 \otimes \pi_2$ exists, is compatible at the unramified finite places, and is of IU-type.  Then the Ramanujan conjecture holds.  That is, if $\pi$ is a unitary cuspidal automorphic representation of $\GL_n(\A_F)$ unramified at a finite place $v$ then $\pi_v$ is tempered.
\end{thm}

It is a classic theorem of Chevalley that given a faithful representation $r:H \to \GL_V$ of an affine algebraic group $H$ over a field, any other irreducible representation of $H$ is a subquotient of $r^{\otimes n} \otimes (r^\vee)^{\otimes m}$ for some $m,n \in \ZZ_{\geq 0}$ \cite[Theorem 4.14]{Milne:AGbook}.  This suggests that the existence of Rankin-Selberg transfers ought to play a distinguished role in the construction of general functorial transfers.  In \S \ref{sec:RS:comments} we explain one way in which this can be made precise.

Let us now outline the structure of this paper. In \S \ref{sec:converse-thm}, we recall the converse theorem that we require. We prove Theorem \ref{thm:RS:intro}, restated as Theorem \ref{thm:RS}, in \S \ref{sec:conv:app}. In \S \ref{sec:Ramanujan:app}, we prove that the Ramanujan conjecture holds under the assumptions of Theorem \ref{thm:Ramanujan:intro}. Finally, in \S \ref{sec:RS:comments}, we explain the relationship between general Langlands transfers and Rankin-Selberg transfers. 

\section*{Acknowledgments}
The authors thank Spencer Leslie for help with highest weight theory.

\section{The converse theorem}\label{sec:converse-thm}

Let us develop the notation necessary to state Cogdell and Piatetski-Shapiro's converse theorem. 
Let $\psi:F \backslash \A_F \to \CC^\times$ be a nontrivial character.  For any triple of irreducible
admissible representations $\pi_i=\otimes_v\pi_{iv}$ of $\GL_{n_i}(\A_F)$  the local $L$-factor and $\varepsilon$-factor 
\begin{align*} 
L(s,\pi_{1v} \times \pi_{2v} \times \pi_{3v})&=L(s,\pi_{1v} \otimes \pi_{2v} \otimes \pi_{3v},\otimes^3),\\
\varepsilon(s,\pi_{1v} \times \pi_{2v} \times \pi_{3v},\psi_v)&=\varepsilon(s,\pi_{1v} \otimes \pi_{2v} \otimes \pi_{3v},\otimes^3, \psi_v)
\end{align*}
 are defined using the local Langlands correspondence \cite{HarrisTaylor:Shimura, Henniart-invent, LanglandsReal, Scholze:p-adic}.  The notation here is standard, see \cite[\S 12.7]{GetzHahn:Book} for instance.  We set
\begin{align*}
L(s,\pi_1 \times \pi_2 \times \pi_3)&=\prod_{v}L(s,\pi_{1v} \times \pi_{2v} \times \pi_{3v}),\\\varepsilon(s,\pi_1 \times \pi_2 \times \pi_3,\psi)&=\prod_v\varepsilon(s,\pi_{1v} \times \pi_{2v} \times \pi_{3v},\psi_v)
\end{align*}
whenever the former product converges.  The latter product always converges because the local factors are equal to $1$ outside finitely many places.
We point out that when $\pi_3$ is the trivial representation $1$ of $\GL_1(\A_F)$ then these factors agree with the usual Rankin-Selberg $L$-factors:
\begin{align*}
    L(s,\pi_{1v} \times \pi_{2v} \times 1)&=L(s,\pi_{1v} \times \pi_{2v}),\\ \varepsilon(s,\pi_{1v} \times \pi_{2v} \times 1,\psi_v)&=\varepsilon(s,\pi_{1v} \times \pi_{2v},\psi_v).
\end{align*}

Let $\pi=\otimes_v\pi_v$ be an irreducible admissible representation of $\GL_n(\A_F)$ with central character $\omega_\pi$ trivial on $F^\times$ and let  $\tau=\otimes_v\tau_v$ be a cuspidal automorphic representation of $\GL_m(\A_F)$, $1\leq m< n$. Then 
$\varepsilon(s, \pi\times \tau,\psi)$ 
is defined, absolutely convergent, and independent of $\psi$ \cite[Lemma 2.1]{Cogdell:PS:I}. We write $\varepsilon(s,\pi \times \tau):=\varepsilon(s,\pi \times \tau,\psi).$ Moreover, if $L(s, \pi)$ is absolutely convergent for $\mathrm{Re}(s)$ sufficiently large, then $L(s, \pi\times\tau)$ and $L(s,\pi^\vee \times \tau^\vee)$ are absolutely convergent for $\mathrm{Re}(s)$ sufficiently large (see \cite[Lemma 2.2, Lemma 2.3]{Cogdell:PS:I}).

The following definition is from \cite{Cogdell:PS:I}:
\begin{defn}\label{defn:nice}
One says that $L(s, \pi\times\tau)$ is \textbf{nice}  if $L(s, \pi\times \tau)$ and $L(s,\pi^\vee\times \tau^\vee)$ have analytic continuations to entire functions of $s$ that are bounded in vertical strips and satisfy 
\begin{align*}
L(s, \pi\times\tau)=\varepsilon(s, \pi\times \tau)L(1-s,\pi^\vee\times \tau^\vee).
\end{align*}    
\end{defn}

We make crucial use of the following converse theorem \cite[Theorem 3]{Cogdell:PS:I}:
\begin{thm}[Cogdell-Piatetski-Shapiro]\label{Thm:3CPS}
Let $n\geq 3$ and let $\pi$ be an irreducible admissible representation of $\GL_n(\A_F)$
whose central character $\omega_\pi$ is trivial on $F^\times$ and whose $L$-function $L(s, \pi)$ is absolutely convergent for $\mathrm{Re}(s)$ sufficiently large. Let $S$ be a finite set of places of $F,$ containing all Archimedean places, such that the ring of $S$-integers has class number one. Suppose that for every $m$ with $1 \leq m <n$ and every cuspidal automorphic representation $\tau$ of $\GL_m(\A_F)$ unramified outside $S,$ the $L$-function $L(s, \pi\times \tau)$ is nice. Then there exists an isobaric automorphic representation $\pi'$ of $\GL_n(\A_F)$ such that $\pi'_v\cong \pi_v$ for all non-Archimedean places $v$ where $\pi_v$ is unramified. \qed
\end{thm} 

Cogdell and Piatetski-Shapiro do not assert that the automorphic representation $\pi'$ is isobaric, but we can always replace $\pi'$ by an isobaric representation without affecting the validity of their theorem. If this is not clear to the reader, the discussion in \cite[\S 10.6-10.7]{GetzHahn:Book} is helpful. 

There are other versions of the converse theorem in the literature, notably those in \cite{Cogdell:PS:ConverseII}.  It seems likely that given the results of loc.~cit. one could restrict attention to $1\leq m \leq (n-2)$ in Theorem \ref{Thm:3CPS}, but this is somewhat irrelevant for our purposes.  We could obtain stronger results using other theorems in \cite{Cogdell:PS:I,Cogdell:PS:ConverseII} if we could show that $\pi_1 \boxtimes \pi_2$ is generic at every place.

\section{Application of the converse theorem} \label{sec:conv:app}

We write $
    \pi_1 \boxtimes \pi_2$
for the admissible representation of $\GL_{n_1n_2}(\A_F)$ that is the transfer of $\pi_1 \otimes \pi_2$ with respect to the tensor product
$$
\otimes^2:{}^L(\GL_{n_1}\times \GL_{n_2}) \lto \GL_{n_1n_2}(\CC).
$$
We refer to $\pi_1 \boxtimes \pi_2$ as the \textbf{Rankin-Selberg transfer} of $\pi_1 \otimes \pi_2.$  Note that it is a priori just an admissible representation. This construction can be iterated to define 
$\pi_1 \boxtimes \cdots \boxtimes \pi_k$ for any collection of irreducible admissible  representations $\pi_i$ of $\GL_{n_i}(\A_F)$ for $1\leq i\leq k.$

If there exists an isobaric automorphic representation $\Pi$  of $\GL_{n_1n_2}(\A_F)$ such that $\Pi_v \cong (\pi_1 \boxtimes \pi_2)_v$ for all $v$ outside of a finite set $S$ of places $F$ then we say that a \textbf{weak Rankin-Selberg transfer} of $\pi_1 \otimes \pi_2$ exists.  We also say that the transfer is \textbf{compatible} with the local Langlands correspondence outside $S.$  If a weak Rankin-Selberg transfer exists, it is unique by strong multiplicity one in the form of \cite[Theorem 4.4]{JacquetShalikaEPII}.

Before stating the main result, we prove the following useful lemma:

\begin{lem}\label{char}
Given a finite set of places $S$ of $F$ one can choose characters $\chi_i:F^\times \backslash \A_F^\times \to \CC^\times$ for $i \in \{1,2\}$ such that if $S_i$ is the set of places that $\chi_i$ is ramified then $S_i \neq \emptyset$ and  $S \cap S_1=S \cap S_2 =S_1 \cap S_2 =\emptyset.$
\end{lem}

\begin{proof}
Let $p_1, p_2 \in \ZZ$ be distinct prime numbers such that $p_i \equiv 1 \pmod{4}.$  Then the character $\chi_i':\QQ^\times \backslash \A_\QQ^\times \to \{\pm 1\}$ attached to the extension $\QQ(\sqrt{p_i})/\QQ$ by class field theory is only ramified at $p_i.$  Now assume that $p_i$ does not divide the discriminant of $F/\QQ$ or any place in $S.$  Letting $\mathrm{N}_{F/\QQ}:F^\times \backslash \A_F^\times \to \QQ^\times \backslash \A_{\QQ}^\times$ be the norm map we see that we may take $\chi_i:=\chi_i' \circ \mathrm{N}_{F/\QQ}.$
\end{proof}

Let $S$ be a finite set of places of $F$ including the infinite places.  

\begin{thm}\label{thm:RS}
Let $\pi_i, i\in \{1,2\},$ be unitary cuspidal automorphic representations of $\GL_{n_i}(\A_F)$ unramified outside $S.$  Assume that
\begin{enumerate}[label=(\alph*)]
\item{For all $1 \leq n_3 < n_1n_2$ and all unitary cuspidal automorphic representations $\pi_3$ of $\GL_{n_3}(\A_F)$ that are ramified at some place not in $S,$ the triple product $L$-function
$L(s,\pi_1\times\pi_2\times\pi_3)$ is nice.}\label{assum:nice}
\item{For all $1 \leq n_3 <n_1n_2$ and all unitary cuspidal automorphic representations $\pi_3$ of $\GL_{n_3}(\A_F),$ the triple product $L$-functions
$L(s,\pi_1\times\pi_2\times\pi_3)$ and $L(s,\pi_1^\vee \times \pi_2^\vee \times \pi_3)$ admit analytic continuations to $\mathrm{Re}(s)>1.$}\label{assum:pole} 
\end{enumerate}
Then there exists a weak Rankin-Selberg transfer of  $\pi_1 \otimes \pi_2.$  It is of IU-type and is compatible at all finite places where both $\pi_1$ and $\pi_2$ are unramified. 
\end{thm}

\begin{proof}
Upon enlarging $S$ if necessary, we assume that $\OO_F^S$ has class number $1.$  Choose characters $\chi_1,\chi_2$ as in Lemma \ref{char} such that the associated places $S_1$ and $S_2$ are both disjoint and disjoint from $S.$

Let $\pi_3$ be a unitary cuspidal automorphic representation of $\GL_{n_3}(\A_F)$ that is unramified outside $S.$ Then $\pi_3 \otimes \chi_i$ is unramified outside $S \cup S_i$ and is ramified at some place in $S_i$.  Therefore by assumption
$
L(s,\pi_1 \times \pi_2 \times (\pi_3 \otimes \chi_i))=L(s,(\pi_1 \times \pi_2 \otimes \chi_i) \times \pi_3)$
is nice.  

The central quasi-character of $\pi_1 \boxtimes \pi_2 \otimes \chi_i$ is the product of $\chi_i$ with the product of the central characters of $\pi_1$ and $\pi_2.$  Hence it is trivial on $F^\times.$  Moreover 
\begin{align}
    L(s,\pi_1 \times (\pi_2 \otimes \chi_i))
\end{align}
converges absolutely for $\mathrm{Re}(s)$ sufficiently large by Rankin-Selberg theory \cite[\S 4]{CogdellPCMI}.
Let $S'$ be the union of the Archimedean places and all places where $\pi_1 \boxtimes \pi_2$ is ramified.
Applying Theorem \ref{Thm:3CPS} we deduce that there is an isobaric automorphic representation $\Pi_i$ of $\GL_{n_1n_2}(\A_F)$ such that  $(\pi_1 \boxtimes (\pi_2 \otimes \chi_i))^{S' \cup S_i} \cong \Pi_i^{S' \cup S_i},$ or equivalently 
\begin{align} \label{Si'isom}
    (\pi_1 \boxtimes \pi_2)^{S' \cup S_i} \cong (\Pi_i \otimes \chi^{-1}_i)^{S' \cup S_i}.
\end{align} 
By strong multiplicity one \cite[Theorem 4.4]{JacquetShalikaEPII} we deduce that 
\begin{align}\label{weak-RS}
\Pi:=\Pi_1 \otimes \chi^{-1}_1 \cong \Pi_2 \otimes \chi^{-1}_2
\end{align}
and hence 
\begin{align} \label{S'isom}
(\pi_1 \boxtimes \pi_2)^{S'} \cong \Pi^{S'}.
\end{align}
Thus a weak Rankin-Selberg transfer exists. 
Since $\pi_1 \boxtimes \pi_2$ is unramified at any place where $\pi_1$ and $\pi_2$ are unramified, we deduce from \eqref{S'isom} that the transfer is compatible at all finite places where $\pi_1$ and $\pi_2$ are unramified.  By replacing $\pi_1$ and $\pi_2$ by their contragredients we see the same is true of the weak Rankin-Selberg transfer of $\pi_1^\vee \otimes \pi_2^\vee.$

We now prove that $\Pi$ defined as \eqref{weak-RS} is of IU-type.  
Write $\Pi=\sigma_1 \boxplus \dots \boxplus \sigma_k$ where the $\sigma_i$ are cuspidal automorphic representations of $\GL_{n_i}(\A_F).$  We must show that the $\sigma_i$ are unitary, or equivalently, have unitary central quasi-character. 

For a cuspidal automorphic representation $\pi$ of $\GL_n(\A_F)$ let $\mathrm{Re}(\pi) \in \RR$ be the unique real number such that 
$$
\pi^u:=\pi \otimes |\det|^{-\mathrm{Re}(\pi)}
$$
is unitary.  Then
$$
L(s,\pi \times (\pi^u)^\vee)=L(s+\mathrm{Re}(\pi),\pi \times \pi^\vee)
$$
is holomorphic apart from simple poles at $s \in \{1-\mathrm{Re}(\pi),-\mathrm{Re}(\pi)\}$ \cite[\S 4.2]{CogdellPCMI}.  Moreover, it is non-vanishing for $\mathrm{Re}(s)\geq 1-\mathrm{Re}(\pi)$ \cite[Theorem 4.3]{CogdellPCMI}.

Order the $\sigma_i$ so that $\mathrm{Re}(\sigma_i)\leq \mathrm{Re}(\sigma_{j})$ for $i \leq j.$  
By assumption \ref{assum:pole}, 
$$
L(s,\Pi \otimes (\sigma_1^u)^\vee)=\prod_{j}L(s,\sigma_j \otimes (\sigma_1^u)^\vee)
$$ 
is holomorphic for $\mathrm{Re}(s)>1.$  
On the other hand, by the facts recalled in the previous paragraph, the factor $L(s,\sigma_1 \otimes (\sigma_1^u)^{\vee})$ has a pole at $1-\mathrm{Re}(\sigma_1),$
and the $j$th factor in the product is nonvanishing for $\mathrm{Re}(s)\geq 1-\mathrm{Re}(\sigma_j)\leq 1-\mathrm{Re}(\sigma_1).$  Thus $0 \leq  \mathrm{Re}(\sigma_1) \leq \mathrm{Re}(\sigma_i)$  for all $i.$  Replacing $\pi_1$ and $\pi_2$  by their contragredients has the effect of replacing $\Pi$ by $\Pi^\vee$ and $\sigma_i$ by $\sigma_i^\vee$ for all $v.$  We deduce that $0 \leq \mathrm{Re}(\sigma_i^\vee) =-\mathrm{Re}(\sigma_i)$ for all $i,$ and hence $\mathrm{Re}(\sigma_i)=0$ for all $i.$  In other words, the $\sigma_i$ are unitary.    
\end{proof}

\section{Application to the Ramanujan conjecture}\label{sec:Ramanujan:app}

\begin{lem}\label{lem-lem1}
Assume that for all $n_1,n_2 \in \ZZ_{\geq 1}$ and all unitary cuspidal automorphic representations $\pi_i$ of $\GL_{n_i}(\A_F),$ $i \in \{1,2\},$ a weak Rankin-Selberg transfer of $\pi_1 \otimes \pi_2$ exists, is compatible at all finite unramified places, and is of IU-type. Then for all $n_1,n_2 \in \ZZ_{\geq 1}$ and all IU-type automorphic representations $\pi_i$ of $\GL_{n_i}(\A_F),$ $i \in \{1,2\},$ a weak Rankin-Selberg transfer of $\pi_1\otimes \pi_2$ exists, is compatible at all finite unramified places, and is of IU-type.  
\end{lem}

\begin{proof}
Write $\pi_1:=\boxplus_i \sigma_i$   and $\pi_2:=\boxplus_j \sigma'_j$ where $\sigma_i$ and $\sigma'_j$ are unitary cuspidal automorphic representations. 
We claim that
 \begin{align} \label{distribute}
\pi_1\boxtimes \pi_2=\boxplus_{i, j}(\sigma_i\boxtimes \sigma'_j).
\end{align}
To check this identity it suffices to work locally.  In view of the local Langlands correspondence, it then suffices to prove the corresponding identity for local $L$-parameters.  Since $\boxtimes$ corresponds to tensor product under the local Langlands correspondence and $\boxplus$ corresponds to direct sum, we deduce the claim.

By assumption, for any finite set of places $S$ including the infinite places such that $\sigma_i^S$ and $\sigma_j'^S$ are unramified we have $(\sigma_i\boxtimes \sigma_j')^S \cong (\boxplus_k\tau_k)^S$ where the $\tau_k$ are unitary cuspidal automorphic representations.  Combining this with \eqref{distribute} we deduce the lemma.
\end{proof}

Let $S$ be a finite set of places of $F$ including the infinite places.
Assume that $\pi_1,\dots,\pi_k$ are a collection of IU-type automorphic representations of $\GL_{n_i}(\A_F)$ unramified outside  $S.$  Then we can inductively form the admissible representation $\pi_1 \boxtimes\dots \boxtimes \pi_k$ of $\GL_{n_1\dots n_k}(\A_F).$  When $\pi_1=\dots=\pi_k$ we write $\pi_1^{\boxtimes k}$ for this representation.  Under the assumptions of Lemma \ref{lem-lem1}, Lemma \ref{lem-lem1} and induction imply that  there is an IU-type automorphic representation $\Pi$ of $\GL_{n_1\dots n_k}(\A_F)$ unramified outside $S$ such that $\Pi^S \cong (\pi_1 \boxtimes \cdots \boxtimes \pi_k)^S.$

\begin{thm}\label{thm:Ramanujan}
Suppose that for all $n_1, n_2 \in \ZZ_{\geq 1}$ and all unitary cuspidal automorphic representations $\pi_i$ of $\GL_{n_i}(\A_F), i\in \{1, 2\},$ a weak Rankin-Selberg transfer of $\pi_1 \otimes \pi_2$ exists, is compatible at the unramified finite places, and is of IU-type.  Then the Ramanujan conjecture holds.  That is, if $\pi$ is a unitary cuspidal automorphic representation of $\GL_n(\A_F)$ unramified at a finite place $v$ then $\pi_v$ is tempered.
\end{thm}
\begin{proof}
Let $S$ be a finite set of places of $F$ including the infinite places.  
Let $\pi$ be a unitary cuspidal automorphic representation of $\GL_n(\A_F)$ that is unramified outside $S.$  

As explained before the statement of the theorem, the assumption of the theorem and Lemma \ref{lem-lem1} imply that there exists an IU-type automorphic representation $\Pi$ such that
$$
(\pi^{\boxtimes k})^S\cong \Pi^S.
$$

Choose a place $v\notin S$ and let $q$ be the order of the residue field of the ring of integers $\OO_{F_v}$ of $F_v.$
For unramified representations $\sigma_v$ of $\GL_m(F_v)$, let $c(\sigma_v) \in \GL_m(\CC)$ be the Langlands class.  
 By \cite[\S 2.5]{JacquetShalikaEPI}, the eigenvalues of $c(\pi_v^{\boxtimes k})$ have complex norm between $q^{-1/2}$ and $q^{1/2}.$  On the other hand, if $\alpha$ is an eigenvalue of $c(\pi_v),$
 then $\alpha^k$ is an eigenvalue of $c(\pi_v^{\boxtimes k}).$
Therefore one has that
$$
q^{-1/2}\leq |\alpha|^k\leq q^{1/2}.
$$
Since $k$ was arbitrary we deduce that $|\alpha|=1.$ In other words $\pi_v$ is tempered.
\end{proof}

\begin{rem} 
We now explain how the hypotheses of Theorem \ref{thm:RS} are implied by the combination of Langlands functoriality and the Ramanujan conjecture.  Assume $\pi_1$ and $\pi_2$ are cuspidal automorphic representations of $\GL_{n_1}(\A_F)$ and $\GL_{n_2}(\A_F)$ that are unitary and satisfy the Ramanujan conjecture.  If the admissible representation  $\pi_1 \boxtimes \pi_2$ is automophic, then $\pi_1 \boxtimes \pi_2 \cong \boxplus_{i=1}^r\sigma_i$ for some unitary cuspidal automorphic representations $\sigma_i$ that also satisfy the Ramanujan conjecture.    Let $\pi_3$ be a unitary cuspidal automorphic representation of $\GL_{n_3}(\A_F).$  By known properties of Rankin-Selberg $L$-functions \cite[\S 4.2]{CogdellPCMI}, the $L$-function $L(s,\pi_1 \times \pi_2 \times \pi_3)$ then has a pole if and only if $\pi_3 \cong \sigma_j^\vee|\det|^{it}$ for some $j$ and $t \in \RR$ and has no poles for $\mathrm{Re}(s)>1.$  
This implies \ref{assum:pole}.  Since the representations $\sigma_i$ are unramified at any place where $\pi_1$ and $\pi_2$ are both unramified, it also implies \ref{assum:nice}.
\end{rem}

\section{On Rankin-Selberg Transfers}\label{sec:RS:comments}

In this section we comment on the relationship between general Langlands transfers and Rankin-Selberg transfers.  
For simplicity, we work with a split group $G$ over the number field $F.$  As usual, let $\widehat{G}$ be the reductive algebraic group over $\CC$ with root datum dual to that of $G.$ 

We consider the problem of constructing Langlands functorial transfers attached to representations
$$
r:{}^LG \lto \GL_V(\CC).
$$
We assume that the restriction to the Galois component is trivial and thereby identify $r$ with an algebraic representation 
$$
r:\widehat{G} \lto \GL_V
$$
of reductive groups over $\CC.$  Choose a maximal torus and Borel subgroup $\widehat{T} \leq \widehat{B}$ of $ \widehat{G}^{\mathrm{der}}.$  Thus we may speak of dominant and highest weights.

For simplicity, we assume that
\begin{enumerate}
\item there are representations $\rho_1,\dots \rho_n$ of $\widehat{G}$ such that the induced representations of $\mathrm{Lie}\,\widehat{G}^{\mathrm{der}}$ are the fundamental representations.  \label{assum1} 
\end{enumerate}
Let $\lambda_i$ be the highest weight of $\rho_i.$

If $\widehat{G}$ is semisimple and simply connected, then \eqref{assum1} is valid \cite[\S 22.8, \S 22.11]{Milne:AGbook}. However, it is not necessary that $\widehat{G}$ be semisimple.  For example, if  $\widehat{G}=\GL_{n+1}$ we may take $\rho_i=\bigwedge^i(V_{\mathrm{st}}),$ where $V_{\mathrm{st}} \cong \GG_a^{n+1}$ is the standard representation of $\GL_{n+1},$ \cite[\S 22.32]{Milne:AGbook}.   

\quash{
\textcolor{cyan}{Let $(G,B,T)$ be a split reductive group, and let $\Phi'$ be the set of positive coroots. By definition $X^*(T)^+:=(\Phi')^\vee\cap X^*(T)=\{\chi\in X^*(T)|\langle\chi,\alpha^\vee\rangle\geq 0,\forall\alpha^\vee\in\Phi'\}$. In particular $X^*(T)^+$ is a rational convex polyhedral cone. By \href{https://en.wikipedia.org/wiki/Gordan\%27s_lemma}{Gordan's lemma} this is automatically a finitely generated monoid.}}

\quash{
\begin{lem}
    Assumption \eqref{span} is true if $\widehat{G}$ is semisimple and simply connected.
\end{lem}
\begin{proof}
One can take the $\lambda_i$ to be the fundamental weights of the Lie algebra.  
Since $\widehat{G}$ is simply connected, each of these are weights of representations $\rho_i:\widehat{G} \to \GL_{V_{\rho_i}}$ . 
\end{proof}}

\begin{defn} \label{defn:virt} Let $S$ be a finite set of places of $F.$ 
    We say that an admissible irreducible representation $\pi^S$ of $\GL_n(\A_F^S)$ is \textbf{stably automorphic} if there exist isobaric automorphic representations $\Pi$ of $\GL_{n+n_0}(\A_F)$ and $\pi'$ of $\GL_{n_0}(\A_F)$ for some $n_0 \geq 0$ such that 
    $$
    \Pi^S \cong \pi^S \boxplus \pi'^S.
    $$
\end{defn}

The notion of a stably automorphic representation appears implicitly in \cite[(4.5)]{JacquetShalikaEPII}. Using the notation of Definition \ref{defn:virt}, if $\pi^S$ is stably automorphic we write
$$
\pi^S=\Pi^S \boxminus \pi'^S.
$$
In many cases we expect that if $\pi^S$ is a stably automorphic representation of $\GL_n(\A_F^S)$ then there is an automorphic representation $\pi_0$ of $\GL_n(\A_F)$ such that $\pi^S_0 \cong \pi^S.$
 However establishing this in any given case appears to be a difficult problem.  We refer to \cite{LYang} for a recent example. 

\begin{lem} \label{lem:diff}
Let $n_1,n_2 \in \ZZ_{\geq 1}$ and let $S$ be a finite set of places of $F.$ If $\Pi$ is an automorphic representation of $\GL_{n_1+n_2}(\A_F)$, $\pi'^S$ is a stably automorphic representation of $\GL_{n_2}(\A_F^S)$ and $\pi^S$ is an irreducible admissible representation of $\GL_{n_1}(\A_F^S)$ such that $\Pi^S \cong \pi'^S\boxplus \pi^S,$ then $\pi^S$ is stably automorphic.
\end{lem}
\begin{proof} Since $\pi'^S$ is stably automorphic,  $\Pi'^S \cong \pi''^S\boxplus \pi'^S$ for some isobaric automorphic representations $\Pi'$ of $\GL_{n_2+n_0}(\A_F)$ (resp. $\pi''$ of $\GL_{n_0}(\A_F)$) and $n_0\geq 0.$ Then
\begin{align*}
    \Pi'^S\boxplus \pi^S \cong \pi''^S\boxplus \pi'^S \boxplus \pi^S \cong \pi''^S\boxplus \Pi^S = (\pi''\boxplus \Pi)^S.
\end{align*}
Since $\pi''\boxplus \Pi$ and $\Pi'$ are isobaric automorphic representations, this implies $\pi^S$ is stably automorphic.
\end{proof}

 If $r:\widehat{G} \to \GL_n$ is a representation, $S$ is a finite set of places of $F$ including the infinite places, and $\pi^S$ is an unramified representation of $G(\A_F^S)$ we denote by $
r(\pi^S):=\otimes_{v \not \in S}r(\pi_v)
$ (see \cite[\S 13.3]{GetzHahn:Book} for this standard notation).  It is an irreducible admissible representation of $\GL_n(\A_F^S).$

\begin{thm} \label{thm:red}
     Let $S$ be a finite set of places of $F$ including the infinite places. Suppose that \eqref{assum1} holds and
    \begin{enumerate}[label=(\alph*)]
        \item \label{weak:RS} weak Rankin-Selberg transfers exist and are compatible at all finite unramified places,
        \item \label{fund:reps} weak transfers of automorphic representations of $G(\A_F)$ to $\GL_{V_{\rho_i}}(\A_F)$ with respect to the representations $\rho_i$ exist and are compatible at all finite unramified places. 
    \end{enumerate}
    Then for any irreducible representation $r:\widehat{G} \to \GL_V$ and any automorphic representation $\pi$ of $G(\A_F)$ unramified outside $S,$ the admissible representation $r(\pi^S)$ is stably automorphic.
\end{thm}

\begin{proof} 
Call two irreducible representations $r:\widehat{G} \to \GL_V$ equivalent if their restrictions to $\widehat{G}^{\mathrm{der}}$ are isomorphic.  If $r_1$ and $r_2$ are equivalent, then there is a character $\chi:\widehat{G}/\widehat{G}^{\mathrm{der}} \to \GG_m$ such that $r_1 =r_2 \otimes \chi.$  Thus the theorem follows for one representation in the equivalence class if and only if it follows for all representations in the equivalence class.  
The equivalence classes of irreducible representations are in natural bijection with dominant weights with respect to the torus $\widehat{T}$ and Borel subgroup $\widehat{B}$ chosen above \eqref{assum1}.

We will use well-founded induction on the highest weights, so let us recall some basic facts.
The set of dominant weights admits the usual associated partial order: $\nu \geq \nu'$ if $\nu-\nu'$ is a nonnegative linear combination of the simple roots.  
  This is a well-founded partial order, that is, every chain admits a minimal element \cite[Corollary 1.4]{Stembridge}.
Moreover, the minimal dominant weights are precisely the minuscule weights and zero \cite[Proposition 1.12]{Stembridge}.  
Every dominant weight 
is of the form
$$
\sum_{i=1}^n m_i\lambda_i
$$
for some $m_i \in \ZZ_{\geq 0}$ \cite[\S 13]{Humphreys}.

Since every minuscule representation is a fundamental representation \cite[Exercise 25.24]{Fulton:Harris} and the equivalence classes of representations of highest  weight $0$ are characters, the basis step of the induction follows from assumption \ref{fund:reps}.

Now consider an irreducible representation $r$ of $\widehat{G}$ whose restriction to $\widehat{G}^{\mathrm{der}}$ has highest weight
$$
\sum_{i=1}^n m_i\lambda_i.
$$
Then $r$ is the unique $\widehat{G}^{\mathrm{der}}$-subrepresentation of 
$$
\rho(m_1,\dots,m_n):=\rho_1^{\otimes m_1} \otimes\dots \otimes \rho_n^{\otimes m_n}
$$
of highest weight $\lambda,$ and every other subrepresentation of $\rho(m_1,\dots,m_n)$ is of strictly lower weight.  If this is not clear to the reader, we point out that it follows from \cite[\S 10.5, Proposition 3]{Procesi}.

Thus, by the first assumption \ref{weak:RS} above and our inductive assumption, there exists an automorphic representation $\Pi$ of $\GL_{V_{\rho(m_1,\dots,m_n)}}(\A_F)$ and a stably automorphic representation $\pi'^S$ such that
$$
\Pi^S \cong r(\pi^S) \boxplus \pi'^S.
$$
This implies the theorem by Lemma \ref{lem:diff}.
\end{proof}

\bibliography{refs}{}
\bibliographystyle{alpha}

\end{document}